\documentclass[12pt]{extarticle}

\usepackage[T1]{fontenc}

\usepackage{amsmath}
\usepackage{amssymb}
\usepackage{amsthm}
\usepackage{bm}

\usepackage{theoremref}
\usepackage[hidelinks, pdfencoding=auto]{hyperref}
\usepackage{cleveref}

\usepackage{enumitem}
\usepackage{abstract}
\usepackage[margin=3.5cm]{geometry}

\usepackage{tikz-cd}
\usepackage{graphicx}
\newcommand{\bigtimes}{\mathop{\vcenter{\hbox{\scalebox{1.8}{$\times$}}}}}

\newtheorem{theorem}{Theorem}[section]
\newtheorem{introtheorem}{Theorem}

\newtheorem{proposition}{Proposition}[section]
\newtheorem{lemma}{Lemma}[section]
\newtheorem{corollary}{Corollary}[section]
\theoremstyle{definition}
\newtheorem{example}{Example}[section]
\newtheorem{remark}{Remark}[section]

\usepackage{titlesec}
\titleformat*{\section}{\large\bfseries\filcenter}
\titleformat*{\subsection}{\normalsize\bfseries}

\newcommand{\ce}{\mathcal{E}}
\newcommand{\cl}{\mathcal{L}}
\newcommand{\cn}{\mathcal{N}}
\newcommand{\co}{\mathcal{O}}
\newcommand{\cq}{\mathcal{Q}}
\newcommand{\cz}{\mathcal{Z}}

\newcommand{\bp}{\mathbf{P}}
\newcommand{\bc}{\mathbf{C}}
\newcommand{\bq}{\mathbf{Q}}

\newcommand{\A}{\mathrm{A}}
\newcommand{\B}{\mathrm{B}} 
\renewcommand{\H}{\mathrm{H}}

\let\SS\S
\renewcommand{\S}{\mathrm{S}}

\newcommand{\Gr}{\operatorname{Gr}}
\newcommand{\Quot}{\operatorname{Quot}}
\newcommand{\Proj}{\operatorname{Proj}}
\newcommand{\Sym}{\operatorname{Sym}}

\begin{document}
	
\title{\textbf{Pl\"{u}cker degrees of Quot schemes}}
\author{SAMUEL STARK}
\date{}
\maketitle

\begin{abstract}
We study the Pl\"{u}cker degree of the main component of the Quot scheme 
of length $l$ quotients of a locally free sheaf on a smooth projective scheme $\S$ of dimension $d\geqslant 1$. 
This degree is determined by classes in the Chow ring of 
the symmetric product $\S^{(l)}$, which are given by the pushforward of the powers of $c_{1}(\co^{[l]})$ 
with respect to the canonical morphism from the Quot scheme to $\S^{(l)}$. 
We describe a decomposition of these classes, 
allowing us to compute the (in a certain sense) leading term of the Pl\"{u}cker degree.
We also obtain a higher-dimensional analogue of a classical result of Schubert.
\end{abstract}

\section*{Introduction}

Let $\ce$ be a locally free sheaf of rank $r$ on a $d$-dimensional smooth projective irreducible scheme $\S$ over $\bc$. 
Grothendieck's Quot scheme $\Quot^{l}_{\S}(\ce)$ of length $l$ quotients of $\ce$ simultaenously generalises 
two well-studied moduli spaces, the Grassmannian and the Hilbert scheme of points. 
In the course of his proof of existence of $\Quot^{l}_{\S}(\ce)$, 
Grothendieck \cite{Gro} defines, for a given sufficiently very ample invertible sheaf $\cl$ on $\S$, 
a closed embedding of $\Quot^{l}_{\S}(\ce)$ into the Grassmannian of $l$-dimensional quotients of $\H^{0}(\ce\otimes\cl)$.
Hence the Plücker embedding
\begin{equation*}
\varpi=\varpi^{l}:\Quot^{l}_{\S}(\ce) \rightarrow \bp(\Lambda^{l} \H^{0}(\ce\otimes\cl))
\end{equation*}
of $\Quot^{l}_{\S}(\ce)$ with respect to $\cl$. 
When $d=0$, $\Quot^{l}_{\S}(\ce)$ coincides with the Grassmannian, and the computation
of the degree of the Plücker embedding is one of the most classical results of enumerative geometry. 
Indeed, after invoking duality, it is a theorem of Schubert \cite{Sch} that
\begin{equation}\label{eq:Catalan}
\deg \varpi^{2} = \frac{ (2r-4)! } { (r-2)! (r-1)!},
\end{equation}
the $(r-2)$-th Catalan number. In \cite{Sch2} he proved, more generally, 
\begin{equation*}
\deg \varpi^{l} = \frac{ (l(r-l))! \prod_{k=1}^{l-1} k! }{ (r-l)! \prod_{k=1}^{l-1} (r-l+k)! }.
\end{equation*}
In this paper, we consider the problem of computing the degree of $\varpi$ for $d\geqslant 1$.

For general $r, d, l$ the scheme $\Quot^{l}_{\S}(\ce)$ can be reducible \cite{Iar,JeS}, 
and neither its dimension nor the number of components are known. 
The most natural way of addressing this issue is to restrict our attention to the main component $\Quot^{l}_{\S}(\ce)_{m}$ of $\Quot^{l}_{\S}(\ce)$,
which is the closure of the locus of quotients of the form $\ce\rightarrow\bigoplus_{i=1}^{l} \co_{s_{i}}$, 
where $s_{1}, \ldots, s_{l}$ are pairwise distinct points of $\S$. 
This is the component of $\Quot^{l}_{\S}(\ce)$ which is most interesting from the point of view of intersection theory; 
for arbitrary rank $r$, $\Quot^{l}_{\S}(\ce)$ is irreducible (and so $\Quot^{l}_{\S}(\ce) = \Quot^{l}_{\S}(\ce)_{m}$) 
when $l\leqslant 3$ or $d\leqslant 2$ \cite{Reg}. 
Denoting by $p=r-1+d$ the dimension of $\bp(\ce)$, 
the dimension of $\Quot^{l}_{\S}(\ce)_{m}$ is $lp$, 
and the Plücker degree of the main component can be written as
\begin{equation*}
\deg\varpi^{l}_{m} = \int_{\Quot^{l}_{\S}(\ce)_{m}} c_{1}(\cl^{[l]})^{lp}.
\end{equation*}
Here $\cl^{[l]}$ is the tautological sheaf associated to $\cl$, 
defined by taking the Fourier-Mukai transform with respect to the universal quotient sheaf.  
Integrals of this type have primarily been studied in the special situation when $\Quot^{l}_{\S}(\ce)$ is smooth:
for $d=1$ \cite{OpP, OpS, BLM}, and, more prominently, for the Hilbert scheme of points $\S^{[l]} = \Quot^{l}_{\S}(\co)$ 
of a surface $\S$ \cite{Leh, EGL, MOP}. 
As $\S^{[l]}$ is also a moduli space of sheaves on $\S$, 
the more general problem of determining the Hilbert polynomial $\chi(\det(\cl^{[l]})^{\otimes n})$ has been considered
\cite{EGL, GoM, Cav}, motivated by the Verlinde formula for curves.

Now the Quot scheme $\Quot^{l}_{\S}(\ce)$ and its main component are typically singular (e.g. when $r, l, d\geqslant 2$), 
and it is well-known that torus localisation, 
which is the key computational technique used in the smooth case, 
is more subtle for singular schemes \cite{EdG}. 
In the smooth case, torus localisation, combined with universality, 
allows one to phrase the problem purely combinatorially, 
while in the singular case it is less clear how to approach it at all 
(and there are much fewer results in this direction \cite{Ren, Ber, Sta}).

Our first result is the higher-dimensional analogue of (\ref{eq:Catalan}).
\begin{introtheorem}\label{firstthm}
For $d\geqslant 1$ we have
\begin{equation*}
\deg\varpi^{2} = 
\frac{1}{2}\binom{2p}{p}\left( \int_{\S} s_{d}(\ce\otimes\cl) \right)^{2} 
- 2^{p-1} \sum_{k=0}^{d} \int_{\S} s_{k}(\S) J_{d-k}(\ce\otimes\cl), 
\end{equation*}
where we have abbreviated
\begin{equation*}
J_{d-k}(\ce\otimes\cl) = \sum_{j=0}^{d-k} \frac{(-1)^{d-k}}{2^{d-j}} P^{p+d-j, -p-k-j}_{r-1+j}(0) s_{d-k-j}(\ce\otimes\cl) s_{j}(\ce\otimes\cl).
\end{equation*}
\end{introtheorem}
Here $P^{p+d-j, -p-k-j}_{r-1+j}(z)$ denotes the Jacobi polynomial. 
The proof uses the birational morphism
\begin{equation*}
\bp(\ce)^{[2]}\rightarrow\Quot^{2}_{\S}(\ce)   
\end{equation*}
constructed by the author for $d=2$ in \cite{Sta}. 
While we expect that a similar construction can be made for $l=3$, 
for higher $l$ this approach does not seem fruitful.

Instead, we consider the symmetric product $\pi:\S^{l}\rightarrow\S^{(l)} = \S^{l}/\mathfrak{S}_{l}$ of $\S$. 
A fundamentally important construction of Grothendieck \cite{Gro} turns $\Quot^{l}_{\S}(\ce)$ into a scheme over $\S^{(l)}$: 
there is a canonical morphism
\begin{equation*}
\mu:\Quot^{l}_{\S}(\ce)\rightarrow\S^{(l)},
\end{equation*}
which takes a point of $\Quot^{l}_{\S}(\ce)$ to the $0$-cycle of the corresponding quotient sheaf. 
For $\ce=\co$, $\mu$ is usually referred to as the Hilbert-Chow morphism, 
and has played a crucial role in the study of Hilbert schemes of points, 
especially for $d=2$.

Restricting to the main component, we define the classes
\begin{equation*}
\mu^{l}_{k}(\ce) = \mu_{\ast} c_{1}(\co^{[l]})^{l(r-1) + k} \in \A^{k}(\S^{(l)}).
\end{equation*}
Here and elsewhere, the Chow groups are understood to be taken with $\bq$-coefficients. 
These classes are reminiscent of the Segre classes --- in fact, $\mu^{1}_{k}(\ce) = (-1)^{k}s_{k}(\ce)$ ---, 
with which they share certain properties; 
when $\cl$ is an invertible sheaf, there is a simple formula for $\mu^{l}_{k}(\ce\otimes\cl)$ in terms of 
the powers of $c_{1}(\cl^{(l)})$ and the $\mu^{l}_{j}(\ce)$ ($0\leqslant j\leqslant k$).
(Here $\cl^{(l)}$ is the invertible sheaf on $\S^{(l)}$ induced by $\cl$ through symmetrization.) 
Our interest in these classes stems from 
\begin{equation*}
\deg\varpi^{l}_{m} = \int_{\S^{(l)}} \mu_{ld}^{l}(\ce\otimes\cl).
\end{equation*}
A natural problem, which refines and is more difficult than the initial numerical problem, 
is then to find an explicit description of the classes $\mu^{l}_{k}(\ce)$.
By comparing $\mu$ to $\bp(\ce)^{(l)}\rightarrow\S^{(l)}$, 
we exhibit a decomposition of $\mu^{l}_{k}(\ce)$ into an explicit class and a class
$\delta^{l}_{k}(\ce)$ coming from the complement $\Delta\subset\S^{(l)}$ of configuration space.
\begin{introtheorem}\label{lastthm}
The class
\begin{equation*}
\delta^{l}_{k}(\ce) =
\mu^{l}_{k}(\ce) - \frac{(-1)^{k}}{l!} \sum_{k_{1}+\cdots+k_{l} = k} 
\frac{(l(r-1)+k)!}{\prod_{m=1}^{l}(r-1+k_{m})!} \pi_{\ast}\bigtimes_{m=1}^{l} s_{k_{m}}(\ce)
\end{equation*}
is in the image of $\A^{k-d}(\Delta)\rightarrow\A^{k}(\S^{(l)})$.
\end{introtheorem}
For degree reasons, this result explicitly determines $\mu^{l}_{k}(\ce)$ for $k<d$, 
as well as $\mu^{l}_{d}(\ce)$ up to a constant. 
It implies
\begin{equation*}
\deg\varpi^{l}_{m} = (-1)^{ld}\frac{(lp)!}{l! p!^{l}}\left(\int_{\S} s_{d}(\ce\otimes\cl)\right)^{l} 
+ \int_{\S^{(l)}} \delta^{l}_{ld}(\ce\otimes\cl),
\end{equation*}
recovering the leading term in \Cref{firstthm}.

\section{Some results of Grothendieck}

\subsection{Notations}

The ideas and results presented in this section come from Grothendieck's paper \cite{Gro}. 
Our notations are essentially the same as in \cite{Sta}, which we briefly recall for the benefit of the reader.
We denote by $\ce$ a locally free sheaf of rank $r$ on a smooth projective irreducible $d$-dimensional scheme $\S$ over $\bc$. 

The Quot scheme $\Quot^{l}_{\S}(\ce)$ is the fine moduli scheme of length $l$ coherent sheaf quotients of $\ce$; 
a point $q$ of $\Quot^{l}_{\S}(\ce)$ thus corresponds to a quotient $\ce\rightarrow\cq_{q}$ of coherent sheaves, 
where $\cq_{q}$ satisfies $\dim(\cq_{q})=0$ and $\dim\H^{0}(\cq_{q})=l$. 
As we have indicated in the introduction, 
we are primarily interested in the main component of $\Quot^{l}_{\S}(\ce)$, 
i.e. the closure $\Quot^{l}_{\S}(\ce)_{m}$ of the locus of points $q$ with $\cq_{q}=\bigoplus_{i=1}^{l} \co_{s_{i}}$. 
It has been understood for a long time \cite{Iar} that the Quot scheme can be very complicated when the dimension $d$ is large, 
and that the main component (sometimes also referred to as the principal or geometric component) is better behaved \cite{Hai, EkS, Ber}; 
for us, it is important that is has the right dimension, i.e. $lp$ with $p=r-1+d$.
(We refer to \cite{JeS} for a study of other components.)

\subsection{Plücker embedding}

Let $\cl$ be an invertible sheaf on $\S$, 
and $\cq$ the universal quotient sheaf on $\S\times\Quot^{l}_{\S}(\ce)$.
The tautological sheaf $\cl^{[l]}$ on $\Quot^{l}_{\S}(\ce)$ is the Fourier-Mukai 
transform of $\cl$ with kernel $\cq$. 
This sheaf is locally free of rank $l$, with fibre 
\begin{equation*}
\cl^{[l]}(q) = \H^{0}(\cl\otimes\cq_{q})  
\end{equation*}
over a point $q$ of $\Quot^{l}_{\S}(\ce)$.  
We do not distinguish notationally between $\cl^{[l]}$ and its restriction to the main component. 
We will frequently use the fact that the formation of $\cl^{[l]}$ is compatible with base change
(in the sense of \cite[Lemma 2.2]{Sta}), 
in particular, that under the canonical isomorphism
\begin{equation}\label{eq:twisting}
\Quot^{l}_{\S}(\ce)\xrightarrow{\sim}\Quot^{l}_{\S}(\ce\otimes\cl)
\end{equation}
taking $\ce\rightarrow\cq_{q}$ to $\ce\otimes\cl\rightarrow\cq_{q}\otimes\cl$, 
the sheaf $\cl^{[l]}$ on $\Quot^{l}_{\S}(\ce)$ corresponds to $\co^{[l]}$ on $\Quot^{l}_{\S}(\ce\otimes\cl)$ \cite[Proposition 2.2]{Sta}. 
It is clear that (\ref{eq:twisting}) takes the main component of $\Quot^{l}_{\S}(\ce)$ to the main component of $\Quot^{l}_{\S}(\ce\otimes\cl)$.

Consider now the canonical morphism of sheaves
\begin{equation*}
\phi_{\cl}:\H^{0}(\ce\otimes\cl)\otimes\co\rightarrow\cl^{[l]}.
\end{equation*}
Grothendieck \cite[\SS 3]{Gro} proved the following result.
\begin{theorem}[Grothendieck]
There exists a very ample invertible sheaf $\cl$ on $\S$ such that 
$\phi_{\cl}$ is surjective and the induced morphism
\begin{equation}\label{eq:grassemb}
\Quot^{l}_{\S}(\ce)\rightarrow\Gr^{l}(\H^{0}(\ce\otimes\cl))
\end{equation}
is a closed immersion.
\end{theorem}
The proof shows that one can take $\cl=\co(n)$ for $n\gg 0$, where $\co(1)$ is a very ample invertible sheaf on $\S$.
Composing (\ref{eq:grassemb}) with the Plücker embedding of $\Gr^{l}(\H^{0}(\ce\otimes\cl))$ defines
the Plücker embedding
\begin{equation*}
\varpi:\Quot^{l}_{\S}(\ce) \rightarrow\bp(\Lambda^{l} \H^{0}(\ce\otimes\cl))
\end{equation*}
of $\Quot^{l}_{\S}(\ce)$ associated to $\cl$. 
By construction, 
\begin{equation*}
\varpi^{\ast}\co(1) = \det(\cl^{[l]}).
\end{equation*}
In particular, the Plücker degree of the main component is
\begin{equation*}
\deg\varpi_{m} = \int_{\Quot^{l}_{\S}(\ce)_{m}}c_{1}(\cl^{[l]})^{lp}.
\end{equation*}
In this paper, we concern ourselves with this integral regardless of any positivity assumptions on $\cl$; 
effectively determining when $\det(\cl^{[l]})$ is very ample is an interesting question in its own right, 
which we do not address.

\subsection{Canonical morphism}

Consider now the symmetric product
\begin{equation*}
\pi:\S^{l} \rightarrow \S^{(l)} = \S^{l} / \mathfrak{S}_{l}.
\end{equation*}
We denote by 
\begin{equation*}
\B(\S, l) \subset \S^{(l)} \quad \mathrm{and} \quad \Delta \subset \S^{(l)} 
\end{equation*}
the configuration space of $l$ unordered points and its complement, respectively. 
For $d\geqslant 2$, $\B(\S, l)$ is the smooth locus of $\S^{(l)}$ \cite[\SS 2]{Fog}.

Grothendieck \cite[\SS 6]{Gro} introduced a canonical morphism 
\begin{equation}\label{eq:norm}
\mu:\Quot^{l}_{\S}(\ce) \rightarrow \S^{(l)},
\end{equation}
taking $q$ to the $0$-cycle $\mu(q)$ of $\cq_{q}$. 
In the Hilbert scheme case, $\mu$ is an isomorphism over $\B(\S, l)$, 
for $d=2$ even a resolution of singularities.
(When $d\geqslant 3$, it is a long-standing problem to construct an 
explicit resolution of singularities of $\S^{(l)}$ \cite[\SS 1]{FuM}.)
We have 
\begin{equation*}
\Quot^{l}_{\S}(\ce)_{m} = \overline{\mu^{-1}(\B(\S, l))}.
\end{equation*}
The construction of $\mu$ involves norms and is somewhat subtle;
for Hilbert schemes a detailed account (which readily generalises to Quot schemes) 
can be found in \cite{EkS}. 

We view $\mu$ as the structural morphism of $\Quot^{l}_{\S}(\ce)$ as a scheme over $\S^{(l)}$;
(\ref{eq:twisting}) is then an isomorphism of schemes over $\S^{(l)}$, 
and it is well-known that $\Quot^{1}_{\S}(\ce)$ is canonically isomorphic to 
$\bp(\ce)=\Proj(\Sym(\ce))$ as a scheme over $\S^{(1)}=\S$, 
with $\co^{[1]}$ corresponding to $\co(1)$.
\section{The degree of \texorpdfstring{$\bm{\varpi^{2}}$}{varpi2}}

\subsection{Relation to the Hilbert scheme of \texorpdfstring{$\bm{\bp(\ce)}$}{P}.}

In this section, we prove \Cref{firstthm}. 
The proof relies on the following construction.
\begin{theorem}
There is a birational morphism of schemes
\begin{equation*}
\rho:\bp(\ce)^{[2]}\rightarrow\Quot^{2}_{\S}(\ce),
\end{equation*}
satisfying $\rho^{\ast}\cl^{[2]} = \cl(1)^{[2]}$.
\end{theorem}
This theorem was proven by the author for a surface \cite[Theorem 3.3]{Sta},
but the construction works (without changes) for any $d\geqslant 1$. 
(The morphism $\rho$ is a resolution of singularities for $d\geqslant 2$; 
it was also used in the recent paper \cite{Pie}.) 
It implies 
\begin{equation}\label{eq:deg2}
\deg\varpi^{2} = \int_{\Quot^{2}_{\S}(\ce)}c_{1}(\cl^{[2]})^{2p} = \int_{\bp(\ce)^{[2]}}c_{1}(\cl(1)^{[2]})^{2p}.
\end{equation}

\subsection{Integrals over \texorpdfstring{$\bm{\S^{[2]}}$}{S2} and \texorpdfstring{$\bm{\bp(\ce)}$}{PE}}

As suggested by (\ref{eq:deg2}), we must first consider the Hilbert scheme case.
\begin{lemma}\label{rank1}
For any invertible sheaf $\cl$ on $\S$
\begin{equation*}
\int_{\S^{[2]}}c_{1}(\cl^{[2]})^{2d} = \frac{1}{2}\binom{2d}{d}\left(\int_{\S}c_{1}(\cl)^{d}\right)^{2} 
- 2^{d-1}\sum_{m=0}^{d}\binom{2d}{d+m} \frac{1}{2^{m}}\int_{\S}c_{1}(\cl)^{d-m}s_{m}(\S).
\end{equation*}
\end{lemma}

\begin{proof}
Consider the flag Hilbert scheme
\begin{equation*}
	\begin{tikzcd}
	& \arrow[dl, swap, "\phi"] \S^{[1,2]} \arrow[dr] &    \\
	\S\times\S	& & \S^{[2]}.
	\end{tikzcd}
\end{equation*}
It is well-known, see \SS 1 and Lemma 4.1 of \cite{FaG}, that the morphism $\phi$ can be identified 
with the blow-up of $\S\times\S$ along the diagonal $\Delta$, 
and that the pullback of $\cl^{[2]}$ with respect to the 
double covering $\S^{[1,2]}\rightarrow \S^{[2]}$ is an extension of $\phi_{2}^{\ast}\cl$ by $\phi_{1}^{\ast}\cl(1)$; 
here $\phi_{i}$ is the composition of $\phi$ with the $i$-th projection $\S\times\S\rightarrow\S$. 
Hence
\begin{equation}\label{eq:FlagIntegral}
2\int_{\S^{[2]}}c_{1}(\cl^{[2]})^{2d} = \sum_{m=0}^{2d}\binom{2d}{m}
\int_{\S^{[1,2]}} \phi^{\ast}c_{1}(\cl\boxplus\cl)^{2d-m} c_{1}(\co(1))^{m}.
\end{equation}
If $\iota$ is the inclusion of the exceptional divisor $\bp(\Omega^{1}_{\S})$ of $\phi$, we have 
\begin{equation*}
c_{1}(\co(1))^{m} = -\iota_{\ast} c_{1}(\co_{\bp(\Omega^{1}_{\S})}(1))^{m-1} \quad (m\geqslant 1).
\end{equation*}
In particular
\begin{equation*}
\phi_{\ast}c_{1}(\co(1))^{m} = -\Delta_{\ast} s_{m-d}(\S) \quad (m\geqslant 1).
\end{equation*}
The only non-vanishing terms in (\ref{eq:FlagIntegral}) are thus $m=0$ and $m\geqslant d$, 
and we are left with
\begin{align*}
2\int_{\S^{[2]}}c_{1}(\cl^{[2]})^{2d} &= \binom{2d}{d}\left(\int_{\S}c_{1}(\cl)^{d}\right)^{2}
-\sum_{m=d}^{2d}\binom{2d}{m}\int_{\S\times\S} c_{1}(\cl\boxplus\cl)^{2d-m}\Delta_{\ast} s_{m-d}(\S) \\
&= \binom{2d}{d}\left(\int_{\S}c_{1}(\cl)^{d}\right)^{2}-\sum_{m=d}^{2d}\binom{2d}{m}2^{2d-m}\int_{\S} c_{1}(\cl)^{2d-m}s_{m-d}(\S). \qedhere
\end{align*}
\end{proof}

Applying \Cref{rank1} to the invertible sheaf $\cl(1)$ on $\bp(\ce)$, 
(\ref{eq:deg2}) gives
\begin{equation}\label{eq:projbundle}
\deg\varpi^{2}
= \frac{1}{2}\binom{2p}{p} I_{0}^{2} 
- 2^{p-1} \sum_{m=0}^{p}\binom{2p}{p+m} \frac{1}{2^{m}} I_{m},
\end{equation}
where we have written
\begin{equation*}
I_{m} = \int_{\bp(\ce)}c_{1}(\cl(1))^{p-m}s_{m}(\bp(\ce)).
\end{equation*}

\begin{lemma}\label{Iintegrals}
We have
\begin{equation*}
I_{m} = \sum_{k=0}^{d} (-1)^{m+k} \int_{\S} s_{k}(\S) 
\left\{ \sum_{j=0}^{d-k} 
(-1)^{j} \binom{r-1+m-k}{m-d+j} s_{d-k-j}(\ce\otimes\cl) s_{j}(\ce\otimes\cl) \right\}.
\end{equation*}
\end{lemma}

\begin{proof}
Using the canonical isomorphism $\bp(\ce) \xrightarrow{\sim} \bp(\ce\otimes\cl)$ of schemes over $\S$, 
we can write $I_{m}$ as
\begin{equation*}
I_{m} = \int_{\bp(\ce\otimes\cl)}c_{1}(\co(1))^{p-m} s_{m}(\bp(\ce\otimes\cl)).
\end{equation*}
Denoting by $f:\bp(\ce\otimes\cl)\rightarrow\S$ the projection, 
the Euler sequence
\begin{equation*}
0\rightarrow\Omega^{1}_{f}\rightarrow f^{\ast}(\ce\otimes\cl)\otimes\co (-1)\rightarrow\co\rightarrow 0
\end{equation*}
and the exact sequence of differentials of $f$ give
\begin{equation*}
s_{m}(\bp(\ce\otimes\cl)) = \sum^{m}_{k=0} (-1)^{m-k} f^{\ast}s_{k}(\S) 
\sum^{m-k}_{n=0} \binom{r-1+m-k}{n} f^{\ast}s_{m-k-n}(\ce\otimes \cl) c_{1}(\co(1))^{n}.
\end{equation*}
Since
\begin{equation*}
f_{\ast} c_{1}(\co(1))^{p-m+n} = (-1)^{d-m+n} s_{d-m+n}(\ce\otimes\cl),
\end{equation*}
we are left with
\begin{equation*}
I_{m} = \sum_{k=0}^{m} (-1)^{d+k} \int_{\S} s_{k}(\S) 
\sum_{n=m-d}^{m-k} (-1)^{n} \binom{r-1+m-k}{n} s_{m-k-n}(\ce\otimes\cl) s_{d-m+n}(\ce\otimes\cl).
\end{equation*}
As the inner sum vanishes for $k>m$, the outer sum may be taken over $0\leqslant k\leqslant d$; 
reindex the inner sum using $j=n-m+d$.
\end{proof}

In particular
\begin{equation*}
I_{0} = (-1)^{d} \int_{\S} s_{d}(\ce\otimes\cl),
\end{equation*}
and to conclude the proof of \Cref{firstthm} 
it remains to rearrange the sum in (\ref{eq:projbundle}).

\subsection{Jacobi polynomials}

By \Cref{Iintegrals}, this sum can be written as
\begin{equation*}
\sum_{m=0}^{p} \binom{2p}{p+m} \frac{1}{2^{m}} I_{m} 
= \sum^{d}_{k=0} \int_{\S} s_{k}(\S) J_{d-k}(\ce\otimes\cl),
\end{equation*}
after changing the order of summation, 
with $J_{d-k}(\ce\otimes\cl)$ of the form
\begin{equation*}
J_{d-k}(\ce\otimes\cl) = \sum_{j=0}^{d-k} a_{j} s_{d-k-j}(\ce\otimes\cl) s_{j}(\ce\otimes \cl).
\end{equation*}
The sums
\begin{equation}\label{eq:jacsum}
a_{j} = (-1)^{k+j} \sum_{m=d-j}^{p} \left(-\frac{1}{2}\right)^{m} \binom{2p}{p+m} \binom{r-1+m-k}{m-d+j}
\end{equation}
are of a special type.

Consider the hypergeometric series
\begin{equation*}
_{2}F_{1}(\alpha, \beta; \gamma; z) = \sum_{m=0}^{\infty} \frac{(\alpha)_{m} (\beta)_{m}}{(\gamma)_{m}} \frac{z^{m}}{m!},
\end{equation*}
where $(x)_{m} = \prod_{i=0}^{m-1}(x+i)$. 
Jacobi \cite[\SS 3]{Jac} studied this series when $\alpha=-n$ is a negative integer, 
in which case it becomes a polynomial of degree $n$.
With the standard normalization, one defines the Jacobi polynomial as
\begin{equation*}
P^{\alpha, \beta}_{n}(z) = \frac{(\alpha+1)_{n}}{n!} {}_{2}F_{1}(-n, n+\alpha+\beta+1; \alpha+1; (1-z)/2).
\end{equation*}
For $\alpha>0$, $\beta>-n-\alpha-1$, we can write this as
\begin{equation*}
P^{\alpha, \beta}_{n}(z) = \sum^{n}_{m=0} \left(\frac{z-1}{2}\right)^{m} \binom{n+\alpha}{m+\alpha} \binom{n+\alpha+\beta+m}{m}.
\end{equation*}
After reindexing, we recognize that (\ref{eq:jacsum}) is given by the constant coefficient of such a polynomial,
\begin{equation*}
a_{j} = \frac{(-1)^{d-k}}{2^{d-j}} P^{p+d-j, -p-k-j}_{r-1+j}(0).
\end{equation*}
\section{The classes \texorpdfstring{$\bm{\mu^{l}_{k}(\ce)}$}{mu}}

\subsection{Definition}

We now associate to $\ce$ classes in the Chow ring of $\S^{(l)}$. 
For $0\leqslant k\leqslant ld$, we define
\begin{equation*}
\mu^{l}_{k}(\ce) = \mu_{\ast} c_{1}(\co^{[l]})^{l(r-1)+k} \in \A^{k}(\S^{(l)}),
\end{equation*}
where it is understood that $\mu$ and $\co^{[l]}$ are restricted to the main component.

Since $\Quot^{1}_{\S}(\ce)=\bp(\ce)$ over $\S$, this construction specializes to the Segre class
\begin{equation*}
\mu^{1}_{k}(\ce) =  (-1)^{k} s_{k}(\ce).
\end{equation*}
The definition of the classes $\mu^{l}_{k}(\ce)$ could hardly be more natural,
and yet they seem to have received very little attention, if at all, 
even in the case of the Hilbert scheme of points $\S^{[l]}$ of a surface $\S$. 

In the latter case,
\begin{equation*}
c_{1}(\co^{[l]}) = -\frac{1}{2}[\partial \S^{[l]}],
\end{equation*}
where $\partial \S^{[l]}$ is the exceptional divisor of $\mu$ \cite[Lemma 3.7]{Leh}. 
The obvious vanishing $\mu_{1}(\co)=0$ appears in the influential paper of Beauville \cite[\SS 9]{Bea}; 
his results imply
\begin{equation*}
\int_{\S^{[l]}} c_{1}(\cl^{[l]})^{2l} = \frac{(2l)!}{l! 2^{l}} \left(\int_{\S}c_{1}(\cl)^{2} + 2 - 2l\right)^{l}
\end{equation*}
when $\S$ is a K3 surface.

It is known, see \cite[Proposition 2.6]{Hai} and \cite[Theorem 7.25]{EkS}, 
that $\mu$ can be identified with the blow up of $\S^{(l)}$
along $\Delta$ (with a specific scheme structure); 
the classes $\mu_{k}(\co)$, $k\geqslant 2$, are thus given by the Segre classes of $\Delta$ in $\S^{(l)}$.

\subsection{Properties} 

Denote by $\cl^{(l)}$ be the invertible sheaf on $\S^{(l)}$ induced by the 
$\mathfrak{S}_{l}$-equivariant invertible sheaf $\cl^{\boxtimes l}$ on $\S^{l}$. 
For our purposes, the two most important properties of $\mu^{l}_{k}(\ce)$ are:

\begin{proposition}\label{muprop}
(i) We have
\begin{equation*}
\deg\varpi_{m} = \int_{\S^{(l)}} \mu_{ld}^{l}(\ce\otimes\cl).
\end{equation*}
(ii) For any invertible sheaf $\cl$ on $\S$
\begin{equation*}
\mu^{l}_{k}(\ce\otimes\cl) 
= \sum^{k}_{j=0} \binom{l(r-1)+k}{k-j} c_{1}(\cl^{(l)})^{k-j} \mu^{l}_{j}(\ce).
\end{equation*}
\end{proposition}

\begin{proof}
(i) This is a consequence of the equality 
\begin{equation*}
\int_{\Quot^{l}_{\S}(\ce)_{m}} c_{1}(\cl^{[l]})^{lp} 
= \int_{\Quot^{l}_{\S}(\ce\otimes\cl)_{m}} c_{1}(\co^{[l]})^{lp}
\end{equation*}
coming from the isomorphism (\ref{eq:twisting}).

As for (ii), it follows from the construction of $\mu$ that 
\begin{equation*}
\mu^{\ast}\cl^{(l)} = \cn_{\cq/\Quot^{l}_{\S}(\ce)}(\cl).
\end{equation*}
Here $\cn_{\cq/\Quot^{l}_{\S}(\ce)}(\cl)$ is the norm, 
in the sense of \cite[\SS 7]{Del}, of (the pullback of) $\cl$ with respect to
the universal sheaf $\cq$ on $\S\times\Quot^{l}_{\S}(\ce)$.
In particular, one has the standard relation
\begin{equation}\label{eq:standard}
\det(\cl^{[l]}) = \mu^{\ast}\cl^{(l)} \otimes \det(\co^{[l]}).
\end{equation}
Since (\ref{eq:twisting}) is an isomorphism of schemes over $\S^{(l)}$,
we have the identity
\begin{equation*}
\mu^{l}_{k}(\ce\otimes\cl) = \mu_{\ast}c_{1}(\cl^{[l]})^{l(r-1)+k},
\end{equation*}
which implies the result by combining (\ref{eq:standard}) with the projection formula. 
\end{proof}

The classes $\mu^{l}_{k}(\ce)$ not only determine the Plücker integrals,
but a somewhat more general class of integrals: by (\ref{eq:standard}) we have
\begin{equation}\label{eq:multint}
\int_{\Quot^{l}_{\S}(\ce)_{m}}\prod_{m=1}^{lp}c_{1}(\cl_{m}^{[l]}) 
= \sum_{k=0}^{ld} \int_{\S^{(l)}} \sigma_{ld-k}(c_{1}(\cl_{1}^{(l)}), \ldots, c_{1}(\cl_{lp}^{(l)})) \mu_{k}^{l}(\ce),
\end{equation}
for any invertible sheaves $\cl_{1}, \ldots, \cl_{lp}$ on $\S$. 
Here $\sigma_{ld-k}$ denotes the elementary symmetric polynomial of degree $ld-k$. 

When the Picard group of $\S$ is infinite cyclic with ample generator $\co(1)$, 
writing $\cl_{m}=\co(n_{m})$, the right-hand side in (\ref{eq:multint}) becomes
\begin{equation*}
\sum_{k=0}^{ld} \sigma_{ld-k}(n_{1}, \ldots, n_{lp}) \int_{\S^{(l)}}  c_{1}(\co(1)^{(l)})^{ld-k} \mu_{k}^{l}(\ce).
\end{equation*}
In this special situation, the integral (\ref{eq:multint}), viewed as a polynomial in $n_{1}, \ldots, n_{lp}$, 
is thus determined by the Plücker integral (effectively through polarization).

\begin{example}
Consider the projective space $\S=\bp_{d}$. Identifying
\begin{equation*}
\A^{\ast}(\bp_{d}^{(l)}) = \A^{\ast}(\bp_{d}^{l})^{\mathfrak{S}_{l}}, 
\quad \A^{\ast}(\bp_{d}^{l}) = \bq[h_{1}, \ldots, h_{l}]/(h_{1}^{d+1}, \ldots, h_{l}^{d+1}),
\end{equation*}
we may view $\mu^{l}_{k}(\ce)$ as a symmetric polynomial of degree $k$ in the truncated variables 
$h_{1}, \ldots, h_{l}$. For $d=1$ one has the peculiarity that $\bp_{1}^{(l)}$ is itself a projective space, 
which implies that $\mu^{l}_{k}(\ce)$ is determined by the Plücker integral.
Indeed, since $\A^{k}(\bp_{1}^{(l)})$ has dimension one,
we can write $\mu^{l}_{k}(\ce) = a_{k}\sigma_{k}(h_{1}, \ldots, h_{l})$ for some $a_{k}\in\bq$. 
Using
\begin{equation*}
k!\sigma_{k}(h_{1}, \ldots, h_{l}) = \sigma_{1}(h_{1}, \ldots, h_{l})^{k},
\end{equation*}
we have
\begin{equation*}
\int_{\Quot^{l}_{\bp_{1}}(\ce)} c_{1}(\co(n)^{[l]})^{lr} = \sum^{l}_{k=0} n^{l-k} \binom{lr}{l-k} \frac{a_{k}}{k!}.
\end{equation*}
Since $\ce$ is split and $\Quot^{l}_{\bp_{1}}(\ce)$ is smooth, 
one can compute the Plücker integral (as a polynomial in $n$, hence the coefficients $a_k$) through torus localisation, as in \cite{OpP,OpS}.
\end{example}

\subsection{Decomposition}

Guided by \Cref{muprop} (i), we seek an explicit description of $\mu^{l}_{k}(\ce)$ 
in terms of the canonical isomorphism
\begin{equation*}
\A^{\ast}(\S^{l})^{\mathfrak{S}_{l}} \xrightarrow{\sim}\A^{\ast}(\S^{(l)})
\end{equation*}
induced by $\pi_{\ast}$.
We first introduce certain auxiliary classes. 
The projection $\bp(\ce)\rightarrow\S$ induces a unique morphism of schemes
\begin{equation*}
\nu:\bp(\ce)^{(l)} \rightarrow \S^{(l)}
\end{equation*}
such that the diagram
\begin{equation}\label{eq:square}
\begin{tikzcd}
\bp(\ce)^{l} \arrow[r] \arrow[d] & \bp(\ce)^{(l)} \arrow[d, "\nu"]   \\
\S^{l} \arrow[r, "\pi", swap] & \S^{(l)}
\end{tikzcd}
\end{equation}
commutes. Parallel to the definition of $\mu^{l}_{k}(\ce)$, we put
\begin{equation*}
\nu^{l}_{k}(\ce) = \nu_{\ast} c_{1}(\co(1)^{(l)})^{l(r-1)+k}.
\end{equation*}
These classes are easy to describe explicitly. 

\begin{lemma}\label{lambdaexplicit}
We have 
\begin{equation*}
\nu^{l}_{k}(\ce) = \frac{(-1)^{k}}{l!} \sum_{k_{1} + \cdots + k_{l} = k} 
\frac{(l(r-1)+k)!}{\prod_{m=1}^{l}(r-1+k_{m})!} \pi_{\ast}\bigtimes_{m=1}^{l} s_{k_{m}}(\ce).
\end{equation*}
\end{lemma}

\begin{proof}
By the projection formula for $\bp(\ce)^{l}\rightarrow\bp(\ce)^{(l)}$, 
$c_{1}(\co(1)^{(l)})^{l(r-1)+k}$ is the pushforward of $c_{1}(\co(1)^{\boxtimes l})^{l(r-1)+k}/l!$. 
The formula results from the multinomial expansion of this latter class and the defining square (\ref{eq:square}).
\end{proof}

Put $\delta^{l}_{k}(\ce) = \mu^{l}_{k}(\ce) - \nu^{l}_{k}(\ce)$.
By the excision exact sequence 
\begin{equation*}
\A^{k-d}(\Delta)\rightarrow\A^{k}(\S^{(l)}) \rightarrow \A^{k}(\B(\S, l))\rightarrow 0,
\end{equation*}
\Cref{lastthm} is equivalent to $\delta^{l}_{k}(\ce)\vert_{\B(\S, l)} = 0 \in \A^{k}(\B(\S, l))$.

\begin{theorem}\label{configspace}
We have
\begin{equation*}
\mu^{l}_{k}(\ce)\vert_{\B(\S, l)} = \nu^{l}_{k}(\ce)\vert_{\B(\S, l)}.
\end{equation*}
\end{theorem}

\begin{proof}
Writing $\mu^{'}:\mu^{-1}(\B(\S, l)) \rightarrow \B(\S, l)$ and 
$\nu^{'}:\nu^{-1}(\B(\S, l)) \rightarrow \B(\S, l)$ for the restrictions 
of $\mu$ and $\nu$ to configuration space, we have
\begin{align*}
\mu^{l}_{k}(\ce)\vert_{\B(\S, l)} 
&= \mu^{'}_{\ast}c_{1}(\co^{[l]}\vert_{\mu^{-1}(\B(\S, l))} )^{l(r-1)+k}, \\
\nu^{l}_{k}(\ce)\vert_{\B(\S, l)} 
&= \nu^{'}_{\ast}c_{1}(\co(1)^{(l)}\vert_{\nu^{-1}(\B(\S, l))} )^{l(r-1)+k}.
\end{align*}
It suffices to construct an isomorphism 
\begin{equation}
\rho:\nu^{-1}(\B(\S, l)) \rightarrow \mu^{-1}(\B(\S, l))
\end{equation}
of schemes over $\B(\S, l)$ satisfying
\begin{equation}\label{eq:invertibles}
\rho^{\ast}c_{1}(\co^{[l]}\vert_{\mu^{-1}(\B(\S, l))}) 
= c_{1}(\co(1)^{(l)}\vert_{\nu^{-1}(\B(\S, l))}).
\end{equation}

Our construction is parallel to the one we used in the proof of \cite[Theorem 3.3]{Sta}.
Since the Hilbert-Chow morphism 
\begin{equation*}
\bp(\ce)^{[l]} \rightarrow \bp(\ce)^{(l)}
\end{equation*}
is an isomorphism over $\B(\bp(\ce), l)$, 
and $\nu^{-1}(\B(\S, l)) \subset \B(\bp(\ce), l)$,
we can view $\nu^{-1}(\B(\S,l))$ as an open subscheme of $\bp(\ce)^{[l]}$, 
whose inclusion we denote by $j$.
Consider the quotient $\co(1) \rightarrow \co_{\cz}(1)$ on $\bp(\ce)\times\bp(\ce)^{[l]}$,
where $\cz\subset\bp(\ce)\times\bp(\ce)^{[l]}$ is the universal subscheme.
Writing $f:\bp(\ce)\rightarrow\S$ for the projection, 
taking $(f\times1)_{\ast}(1\times j)^{\ast}$ of the quotient $\co(1) \rightarrow \co_{\cz}(1)$, 
we obtain a flat quotient of sheaves
\begin{equation*}
\ce\rightarrow (f\times1)_{\ast}(1\times j)^{\ast}\co_{\cz}(1)
\end{equation*}
on $\S\times\nu^{-1}(\B(\S,l))$, 
and a corresponding morphism
\begin{equation}\label{eq:rho}
\nu^{-1}(\B(\S,l)) \rightarrow \Quot^{l}_{\S}(\ce)
\end{equation}
Identifying $\bp(\ce)=\Quot^{1}_{\S}(\ce)$, 
a point in $\nu^{-1}(\B(\S,l))$ is given by $l$ quotients
$\ce\rightarrow\co_{s_{i}}$ for pairwise distinct points $s_{1}, \ldots, s_{l}$ of $\S$; 
the morphism (\ref{eq:rho}) takes this point to the quotient $\ce\rightarrow\bigoplus_{i=1}^{l} \co_{s_{i}}$ of $\Quot^{l}_{\S}(\ce)$. 
Since any point of $\mu^{-1}(\B(\S, l))$ is of this form, (\ref{eq:rho}) induces a bijection onto $\mu^{-1}(\B(\S, l))$ and hence 
(since both $\nu^{-1}(\B(\S,l))$ and $\mu^{-1}(\B(\S, l))$ are smooth) the isomorphism $\rho$. 
It is clear from our description that $\rho$ is an isomorphism over $\B(\S, l)$, 
and by using the base change property of tautological sheaves \cite[Lemma 2.2]{Sta} for the Cartesian squares
\begin{equation*}
\begin{tikzcd}
\S\times\nu^{-1}(\B(\S, l))  \arrow[r] \arrow[d] & \S\times\Quot^{l}_{\S}(\ce) \arrow[d]   \\
\nu^{-1}(\B(\S, l)) \arrow[r] & \Quot^{l}_{\S}(\ce)
\end{tikzcd}
\end{equation*}
and 
\begin{equation*}
\begin{tikzcd}
\bp(\ce)\times\nu^{-1}(\B(\S, l)) \arrow[r] \arrow[d] & \bp(\ce)\times\bp(\ce)^{[l]} \arrow[d]   \\
\nu^{-1}(\B(\S, l)) \arrow[r] & \bp(\ce)^{[l]},
\end{tikzcd}
\end{equation*}
we obtain a canonical isomorphism
\begin{equation*}
\rho^{\ast}\co^{[l]}\vert_{\mu^{-1}(\B(\S, l))} \xrightarrow{\sim} \co(1)^{[l]}\vert_{\nu^{-1}(\B(\S, l))}.
\end{equation*}

Since the universal family of $\bp(\ce)^{[l]}$ is étale over $\B(\bp(\ce), l)$, 
we have the vanishing $c_{1}(\co^{[l]} \vert_{\B(\bp(\ce), l)}) = 0$ (over $\bq$). 
Using (\ref{eq:standard}) for $\bp(\ce)^{[l]}\rightarrow\bp(\ce)^{(l)}$, 
we obtain 
\begin{equation*}
c_{1}(\co(1)^{[l]}\vert_{\B(\bp(\ce), l)}) = c_{1}(\co(1)^{(l)}\vert_{\B(\bp(\ce), l)}),
\end{equation*}
which implies (\ref{eq:invertibles}) as $\nu^{-1}(\B(\S, l)) \subset \B(\bp(\ce), l)$.
\end{proof}

\begin{corollary}\label{munueq}
\begin{equation*}
\mu^{l}_{k}(\ce) = \nu^{l}_{k}(\ce) \quad (0\leqslant k\leqslant d-1)
\end{equation*}    
\end{corollary}

Also, since $\Delta\subset\S^{(l)}$ is irreducible,
\begin{equation*}
\delta^{l}_{d}(\ce) = c\pi_{\ast}\sum_{i<j}[\Delta_{ij}]   
\end{equation*}
for some constant $c$, where $\Delta_{ij}\subset\S^{l}$ is the locus of $(s_{1}, \ldots, s_{l})$ with $s_{i}=s_{j}$.
\begin{corollary}
\begin{equation*}
\mu^{l}_{k}(\co^{\oplus r}) = 0 \quad (1\leqslant k\leqslant d-1)
\end{equation*}   
\end{corollary}

Indeed, $\nu^{l}_{k}(\co^{\oplus r}) = 0$ for $k\geqslant 1$ by \Cref{lambdaexplicit}.

\begin{example}\label{lowmu}
We have
\begin{equation*}
\mu^{l}_{0}(\ce) = \frac{(l(r-1))!}{(r-1)!^{l}}
\end{equation*}
as well as, using $c_{1}(\ce)=c_{1}(\det(\ce))$ and $\pi_{\ast}c_{1}(\det(\ce)^{\boxtimes l}) = l! c_{1}(\det(\ce)^{(l)})$,
\begin{equation*}
\mu^{l}_{1}(\ce) = \frac{(l(r-1)+1)!}{(r-1)!^{l-1}r!}c_{1}(\det(\ce)^{(l)}) \quad (d\geqslant 2).
\end{equation*}
\end{example}

\begin{remark}
(i) The twisting property of \Cref{muprop} (ii) is also satisfied by the classes $\nu^{l}_{k}(\ce)$, 
hence also by $\delta_{k}^{l}(\ce)$.

(ii) Given a very ample invertible sheaf $\co(1)$, let $\cl=\co(n)$. Then
\begin{equation*}
\deg\varpi_{m} = \int_{\S^{(n)}}\mu^{l}_{ld}(\ce(n))
= \sum^{ld}_{j=0} \binom{lp}{ld-j} n^{ld-j} \int_{\S^{(l)}} 
c_{1}(\co(1)^{(l)})^{ld-j} \mu^{l}_{j}(\ce),
\end{equation*}  
using \Cref{muprop} (ii). 
\Cref{munueq} thus determines the top $d-1$ coefficients of $\deg\varpi_{m}$ as a polynomial in $n$.
\end{remark}

We also obtain
\begin{align*}
\deg\varpi_{m} &= \int_{\S^{(l)}} \nu_{ld}^{l}(\ce\otimes\cl) + 
\int_{\S^{(l)}} \delta^{l}_{ld}(\ce\otimes\cl) \\
&= (-1)^{ld}\frac{(lp)!}{l!(p)!^{l}}\left(\int_{\S} s_{d}(\ce\otimes\cl)\right)^{l} 
+ \int_{\S^{(l)}} \delta^{l}_{ld}(\ce\otimes\cl).
\end{align*}
The exact nature of the classes $\delta_{k}^{l}(\ce)$ is not clear to us;
comparing the above equation with \Cref{firstthm}, 
we see that $\delta^{l}_{ld}(\ce\otimes\cl)$ must be fairly complicated.

We expect that there should be an explicit formula 
for the classes $\delta_{k}^{l}(\ce)$ in terms of the Segre classes of $\ce$ and $\S$, 
and diagonal classes. 
The latter classes have favourable properties when $\S$ is a K3 surface \cite{BeV} 
(see also \cite{Voi} and \cite{MaN}); 
for $d=1$ it may also be interesting to consider the corresponding problem in cohomology. 
(The cohomology of $\Quot^{l}_{\S}(\ce)$ has been studied in \cite{Mar}). 

\subsection*{Acknowledgments}

This work was partially supported by the Simons Foundation (grant 817244), 
the EPSRC (grant EP/R014604/1), and the ERC (grant 786580).

\end{document}